\documentclass[12pt]{article}

\usepackage[margin=1in]{geometry}

\usepackage{amsthm,amsmath,amssymb}

\usepackage{enumerate}

\usepackage{graphicx}

\usepackage[colorlinks=true,citecolor=black,linkcolor=black,urlcolor=blue]{hyperref}

\newcommand{\vol}{{\rm Vol}}
\newcommand{\e}{\mathbb{E}}
\newcommand{\p}{\mathbb{P}}

\theoremstyle{plain}
\newtheorem{theorem}{Theorem}
\newtheorem{lemma}[theorem]{Lemma}
\newtheorem{cor}[theorem]{Corollary}
\newtheorem{prop}[theorem]{Proposition}

\theoremstyle{definition}

\theoremstyle{remark}

\title{\bf Many edge-disjoint rainbow spanning trees\\ in general graphs}

\author{Paul Horn\thanks{Research supported in part by NSA Young Investigator Grant H98230-15-1-0258.}\\
\small Department of Mathematics\\[-0.8ex]
\small University of Denver\\[-0.8ex] 
\small Denver, CO, U.S.A.\\
\small\tt Paul.Horn@du.edu\\
\and
Lauren M. Nelsen\\
\small Department of Mathematics\\[-0.8ex]
\small University of Denver\\[-0.8ex]
\small Denver, CO, U.S.A.\\
\small\tt Lauren.Nelsen@du.edu
}

\date{}

\begin{document}

\maketitle

\begin{abstract}
A rainbow spanning tree in an edge-colored graph is a spanning tree in which each edge is a different color.  Carraher, Hartke, and Horn showed that for $n$ and $C$ large enough, if $G$ is an edge-colored copy of $K_n$ in which each color class has size at most $n/2$, then $G$ has at least $\lfloor n/(C\log n)\rfloor$ edge-disjoint rainbow spanning trees.  Here we strengthen this result by showing that if $G$ is \emph{any} edge-colored graph with $n$ vertices in which each color appears on at most $\delta\cdot\lambda_1/2$ edges, where $\delta\geq C\log n$ for $n$ and $C$ sufficiently large and $\lambda_1$ is the second-smallest eigenvalue of the normalized Laplacian matrix of $G$, then $G$ contains at least $\left\lfloor\frac{\delta\cdot\lambda_1}{C\log n}\right\rfloor$ edge-disjoint rainbow spanning trees.

\end{abstract}

\section{Introduction}

For an edge-colored graph $G$, a \textit{rainbow spanning tree} of $G$ is a spanning tree in which each edge is a different color.  Here we show that for {\it any} edge-colored graph $G$ on sufficiently many vertices and with large enough minimum degree, we can give a lower bound on the number of edge-disjoint rainbow spanning trees in $G$.  This lower bound will depend on an eigenvalue of a matrix associated with $G$ known as the normalized Laplacian matrix, which we will define later.

Our motivation is a conjecture of Brualdi and Hollingsworth, which says if $K_n$ (for $n\geq 6$ and $n$ even) is edge-colored such that each color class is a perfect matching, then there is a decomposition of the edges into $n/2$ edge-disjoint rainbow spanning trees (\cite{BrualdiHollingsworth}).  Progress was slow: they proved that in any such edge-colored $K_n$, there are at least two edge-disjoint rainbow spanning trees.  Krussel, Marshall, and Verrall showed that there are at least three edge-disjoint rainbow spanning trees (\cite{KMV}).  Horn (\cite{Horn}) showed that under these hypotheses, a postitve fraction of the graph can be covered by edge-disjoint rainbow spanning trees.  This is the best known result for the conjecture of Brualdi and Hollingsworth.  Recently, Fu, Lo, Perry, and Rodger gave a constructive proof that in a properly edge-colored $K_n$ (where $n$ is even) there is a decomposition of the edges into at least $\lfloor\sqrt{3n+9}/3\rfloor$ edge-disjoint rainbow spanning trees (\cite{FuLoPerryRodger}).  

Strengthening the conjecture of Brualdi and Hollingsworth, Kaneko, Kano, and Suzuki conjectured that if $G$ is a properly edge-colored $K_n$ where $n\geq 6$ and $n$ is even, then there is a decomposition of the edges into at least $\lfloor n/2\rfloor$ edge-disjoint rainbow spanning trees (\cite{KKS}).  Also related, Akbari and Alipour showed in \cite{AkbariAlipour} that if $G$ is an edge-colored $K_n$ ($n\geq 5$) in which each color appears at most $n/2$ times, then $G$ contains at least two edge-disjoint rainbow spanning trees.  Carraher, Hartke, and Horn showed in \cite{CHH} that for $n$ and $C$ sufficiently large, if $G$ is an edge-colored copy of $K_n$ in which each color appears less than $n/2$ times, then $G$ contains at least $\lfloor n/(C\log n)\rfloor$ edge-disjoint rainbow spanning trees.  This result is currently the best known for the conjecture of Kaneko, Kano, and Suzuki.  

There are a number of results about rainbow structures other than spanning trees in edge-colored graphs.  Kano and Li did a survey of many results and conjectures about such structures in \cite{KanoLi}.  Brualdi and Hollingsworth looked at edge-colored complete bipartite graphs and proved results about when such graphs contain rainbow forests or trees (\cite{BrualdiHollingsworth2}).  Constantine showed that for $p$ prime ($p>2$), there is some proper edge-coloring of the complete graph $K_p$  such that there is a partition of the edges of $K_p$ into rainbow hamiltonian cycles (\cite{Constantine}).  He also showed that for certain values of $n$, there is a proper edge-coloring of $K_n$ such that there is a partition of the edges of $K_n$ into isomorphic rainbow spanning trees.  Rainbow cycles in graphs have also been studied.  Albert, Frieze, and Reed showed that for $n$ sufficiently large, if $K_n$ is edge-colored such that each color appears less that $n/64$ times, then there is a rainbow hamiltonian cycle (\cite{AlbertFriezeReed}).  (Rue gave a correction of this constant -- see \cite{FriezeKrivelevich}.)  Frieze and Krivelevich  proved that there is a constant $c>0$ such that an edge-coloring of $K_n$ in which each color appears at most $\max\{cn,1\}$ times contains a rainbow cycle of length $k$ for each $3\leq k\leq n$ (\cite{FriezeKrivelevich}).

For a vertex, $v$, let $d_v$ denote the degree of $v$.  The normalized Laplacian of a graph $G$ is the matrix $\mathcal{L}$, whose entries are given by the following:

$$\mathcal{L}(u,v)=\begin{cases}  
1 & \text{ if } u=v\text{ and } d_v\neq 0\\
-\frac{1}{\sqrt{d_v\cdot d_u}} & \text{ if } u\sim v \\
0 & \text{ otherwise.}
\end{cases}$$

(The monograph of Chung, \cite{Chung}, gives significantly more information about the normalized Laplacian.)  The eigenvalues of $\mathcal{L}$ are $0=\lambda_0\leq \lambda_1\leq\cdots\leq \lambda_{n-1}\leq 2$.  We show the following:

\begin{theorem} \label{t1}
If $G$ is an edge-colored graph with minimum degree $\delta\geq C\log n$ (for $C$ and $n$ sufficiently large) in which each color class has size at most $\delta\cdot\lambda_1/2$, then $G$ contains at least $\left\lfloor\frac{\delta\cdot\lambda_1}{C\log n}\right\rfloor$ edge-disjoint rainbow spanning trees.  
\end{theorem}

We emphasize that our theorem works for both regular and irregular graphs; use of spectral methods frequently restrict results to apply only for irregular graphs.  We also emphasize that the colorings considered in our theorem need not be proper -- there is only a restriction on the multiplicity of a color.  Another advantage of our approach is that it uses only the smallest eigenvalue.  Results in extremal combinatorics using spectral graph theory frequently assume strong control on both the smallest (non-trivial) and largest eigenvalue of $\mathcal{L}$ as such gives stronger pseudo-random properties of the edge set of a graph via the expander mixing lemma.  We also note that our result does not actually require $\lambda_1$ to be close to $1$ (another common requirement), although our result is certainly strongest if $\lambda_{1}$ is close to $1$.  Some additional comments regarding the hypothesis of our results are given in Section \ref{sec5}.   

The paper is organized as follows:  Sections \ref{sec:db} and \ref{sec3} introduce definitions and preliminary results.  The proof of Theorem \ref{t1} is in Section \ref{sec4}, and we conclude in Section \ref{sec5} with some discussion, along with some applications of Theorem \ref{t1} to particular classes of graphs where it yields particularly strong results.

\section{Definitions and Background} \label{sec:db}

For $k,\ell\in\mathbb{N}$, $[k]$  denotes the set $\{1,...,k\}$, and ${[k]\choose \ell}$ the collection of all subsets of $[k]$ of size $\ell$.  Let $q=\lfloor \delta\cdot\lambda_1/(C\log n)\rfloor$  where $n$ and $C$ are sufficiently large.  Our proof proceeds by constructing $G_1,...,G_q$ as follows:  for each edge in $G$, independently and uniformly select a $G_i$ with probability $1/q$, and add that edge to $G_i$.  Then the set $\{G_1,...,G_q\}$ forms a partition of the edges of $G$.  We show that $G_1,...,G_q$ each contain a rainbow spanning tree with high probability.  For a subset $S\subseteq V(G)$, let $e_j(S,\overline{S})$ denote the number of edges in $G_j$ with one end in $S$ and the other in $\overline{S}$.


We will let $C_1,...,C_s$ be the color classes, and for each $i\in [s]$, let $c_i=|C_i|$.  Assume $\delta\geq C\log n$, $p=\frac{C\log n}{\delta\cdot\lambda_1}$ and for each $i\in [s]$, $1\leq c_i\leq\frac{\delta\cdot\lambda_1}{2}$.  In order to show that each $G_i$ has a rainbow spanning tree, we use the following proposition, originally due to Schrijver \cite{Schrijver}.

\begin{prop}\label{RSTexistence}
A graph $G$ has a rainbow spanning tree if and only if for every partition $\mathcal{P}$ of $V(G)$ into $t$ parts, there are at least $t-1$ different colors represented between the parts of $\mathcal{P}$.
\end{prop}

Broersma and Li (\cite{BroersmaLi}) showed that the Matroid Intersection Theorem (\cite{Edmonds}) can be used to determine the largest rainbow spanning forest in a graph.  (See \cite{Schrijver}.)  Schrijver (\cite{Schrijver}) showed that the conditions of the Matroid Intersection Theorem are equivalent to the necessary and sufficient conditions from Proposition \ref{RSTexistence} for the existence of a rainbow spanning tree.  Suzuki (\cite{Suzuki}) and Carraher and Hartke (\cite{CarraherHartke}) provided additional graph theoretical proofs of this result.

Our strategy is to take our random partition of the edges of $G$ and prove some structural results that hold with high probability (Lemma \ref{ColorClassSize} below).  Then we show deterministically that each graph satisfies Proposition \ref{RSTexistence}.  The strategy is similar to that of \cite{CHH}, with additional technical difficulties given from the fact that our underlying graph is not complete and, instead, we only have spectral information to understand the geometry of the host graph $G$.  In some sense our primary new difficulty is to extract sufficient geometric information from the spectrum to push the analysis through.   


We will frequently use the fact that if $A_1,...,A_{\ell}$ are events, then
$$ \mathbb{P}\left[ \bigcup_{i=1}^{\ell} A_{\ell} \right]\leq\sum_{i=1}^{\ell}\mathbb{P}[A_i]. $$

We also use the following Chernoff bounds.

\begin{lemma}(\cite{Chernoff})\label{Chernoff}
If $X_i$ are independent random variables with 
$$\mathbb{P}(X_i=1)=p_i,\text{      } \mathbb{P}(X_i=0)=1-p_i$$
and $X=\sum_i X_i$, ($\mathbb{E}[X]=\sum_{i}p_i$) then
$$ \mathbb{P} [ X\leq\mathbb{E}[X]-\lambda ]\leq\exp\left( -\frac{\lambda^2}{2\mathbb{E}[X]} \right) $$
and
$$ \mathbb{P}[X\geq \mathbb{E}[X]+\lambda]\leq\exp\left( -\frac{\lambda^2}{2(\mathbb{E}[X]+\lambda/3)} \right). $$
\end{lemma}

For a graph $G$, the volume of a subset $X\subset V(G)$, denoted by $\vol(X)$, is defined as follows: 
$$ \vol(X)=\sum_{v\in X}\deg(v). $$
For a subset $S$, of $V(G)$, we define $h_G(S)=\frac{|E(S,\overline{S})|}{\min\{ \vol(S),\vol(\overline{S}) \}}$.  The \textit{Cheeger constant} (or \textit{isoperimetric constant}), $h_G$ is then defined by  $$ h_G=\min_{S} h_G(S). $$

Determining $h_g$ is computationally difficult, but is related to the smallest eigenvalue through the following result known as \emph{Cheeger's inequality}.

\begin{theorem}(\cite{Chung})  If $G$ is a connected graph and $\lambda_1$ is the second-smallest eigenvalue of the normalized Laplacian of $G$, then

$$ \frac{h_G^2}{2}<\lambda_1\leq 2h_G. $$

\end{theorem}


\section{Preliminary Results} \label{sec3} 

We proceed by proving some preliminary results.  We begin by establishing some properties of each of the graphs $G_j$ constructed above.  

\begin{lemma}\label{ColorClassSize}
Fix $\epsilon=0.1$.  For every $j\in [q]$, the edge sets, $E_j$ of $G_j$ satisfy

\begin{enumerate}[(i)]
\item For every $i\in [s]$, $|E_j\cap C_i|\leq (1+\epsilon)\frac{C\log n}{2}$

\item [(ii)] For every set $S \subseteq V(G)$ with $\vol(S) \leq \frac{1}{2}\vol(G)$,
$$ e_j(S,\bar{S})\geq (1-\epsilon)\e[e(S,\bar{S})] $$
\item[(iii)] For every vertex $v \in V(G)$, 
\[
\deg_{G_j}(v) \geq (1-\epsilon) C \log n
\] 
\end{enumerate}
simultaneously with probability at least  $1-n^{-2}$, assuming $n$ is sufficiently large.
\end{lemma}

{\bf Remark:} The choice of $\epsilon = 0.1$ is not too important -- any sufficiently small $\epsilon$ will suffice.  

\begin{proof}
Fix a color $i\in [s]$ and let $\epsilon = 0.1$. 

To prove (i), note that $\mathbb{E}[|E_j\cap C_i|]=p\cdot c_i\leq C\cdot\log n$.  Using Lemma \ref{Chernoff} with $\lambda= \epsilon\cdot\frac{C\log n}{2} $ implies that
\begin{align*}
\mathbb{P}\left( |E_j\cap C_i|\geq (1+\epsilon)\frac{C\log n}{2} \right)&\leq \exp\left( -\frac{\epsilon^2\cdot C\cdot \log n}{2(1+\epsilon/6)} \right)\\
 &\leq \exp\left( -\frac{\epsilon^2\cdot C\cdot\log n}{3} \right)\\
 &\leq n^{-5}\text{ for } C\geq \frac{15}{\epsilon^2}.
\end{align*}

Part $(iii)$ is merely a (useful) special case of $(ii)$, so it suffices to prove $(ii)$.  
To prove $(ii)$, first fix a set $S$ of size $k$ and with volume at most $\frac{1}{2} \vol(G)$.  Then, by Cheeger's inequality, $e(S,\bar{S}) \geq \frac{\lambda_1}{2} \vol(S) \geq \frac{\lambda_1 \delta k}{2}$.  Hence $\e[e_j(S, \bar{S})] \geq \frac{Ck \log(n)}{2}$.  Applying the Chernoff bounds with $\lambda = \epsilon \cdot \e[e_j(S,\bar{S})]$ yields
\begin{align}
\p\Big(e_j(S,\bar{S}) \leq (1-\epsilon) \e[e_j(S,\bar{S})]\Big) &\leq \exp\left(-\frac{\epsilon^2}{2} \e[e_j(S,\bar{S})]\right) \nonumber \\
&\leq \exp\left(-\frac{C\epsilon^2}{4} k \log n\right). \label{eq:cher}
\end{align}
Let $\mathcal{B}$ denote the event that there exists a set $S$ which doesn't satisfy the conclusion of part (ii).  A union bound over $k$ and $S$ of size $k$ yields
\begin{align*}
\p(\mathcal{B}) &\leq \sum_{k=1}^{n} \sum_{\substack{S: |S| = k \\ \vol(S) \leq \vol(G)/2}} \p\Big(e_j(S,\bar{S}) \leq (1-\epsilon) \e[e_j(S,\bar{S})]\Big) \\ 
&\leq  \sum_{k=1}^{n} {n \choose k} \exp\left(-\frac{C\epsilon^2}{4} k \log n\right) \\
&\leq \sum_{k=1}^{n} \exp\left(-\left(\frac{C\epsilon^2}{4} - 1\right) k \log n\right)\\
&\leq n^{-4}.
\end{align*} 
Here the second inequality follows from \eqref{eq:cher} and the fact that there are at most ${n \choose k}$ sets of size $k$ satisfying $\vol(S) \leq \vol(G)/2$, the third from the simple bound that ${n \choose k} \leq n^k$, and the last inequality holds assuming that $C$ is sufficiently large.

A union bound over all $j \in [q]$ and all color classes $i \in [s]$ yields the result.

\end{proof}


Lemma \ref{ColorClassSize} provides lower bounds on the number of edges leaving a set, and upper bounds on the number of edges in a particular color in each of our graphs $G_j$.  In order to apply Proposition \ref{RSTexistence} to then prove that the graphs contain rainbow spanning trees, we thus must study the number of edges between parts.  This requires some care.  

Suppose $\mathcal{P} = \{P_1, P_2, \dots,P_t\}$ is partition of $V(G)$ into $t$ parts.  We use the notation 
\[
e_j(\mathcal{P}) = \frac{1}{2} \sum_{i} e_j(P_i,\bar{P}_i)
\]
to denote the total number of edges between parts in the graph $G_j$.    (We denote the number of edges between parts in $G$ by $e_G(\mathcal{P})$.)

The following Lemma is then immediate,
\begin{lemma}\label{enoughedges}
Suppose $\mathcal{P}$ is a partition of $V(G)$ into $t$ parts, Lemma \ref{ColorClassSize} (i) is satisfied and 
$$e_j(\mathcal{P}) \geq (t-2)(1+\epsilon)\frac{C\log n}{2}+1,$$
where $\epsilon = 0.1$ as in Lemma \ref{ColorClassSize}.  Then there will be at least $t-1$ colors between parts of $\mathcal{P}$ in $G_j$.
\end{lemma}


Using the Cheeger inequality to lower bound the number of edges leaving a set proves insufficient for our goals, at least for small sets.  For a set $S\subseteq V(G)$, let \begin{equation} 
f(S)=\max\left\{ \frac{\lambda_1}{2}\vol(S),\vol(S)-2{|S|\choose 2} \right\}.  \label{eq:f}
\end{equation}
Both quantities serve as a lower bound for $e(S,\bar{S})$ and hence,
\[
\mathbb{E}[e_j(S,\overline{S})]\geq p\cdot f(S).
\]
As a more convenient way of applying the lower bound from \eqref{eq:f}, we let
\begin{equation}
g(z) = \max\left\{ \frac{\lambda_1 \delta z}{2}, \delta z - z(z-1)\right\}. \label{eq:g}
\end{equation}
As both quantities in the right hand side of \eqref{eq:f} are monotone increasing with $\vol(S)$ and $\vol(S) \geq \delta |S|$, we thus have that $f(S) \geq g(|S|)$. 

In order to be able to apply Lemma \ref{enoughedges} to verify the hypothesis of Proposition \ref{RSTexistence}, we need to be somewhat careful when minimizing the number of edges crossing a partition.  We accomplish this as follows:

\begin{lemma}\label{EdgesConcavity}
Let $\mathcal{P}=\{ P_1,...,P_t \}$ be a partition of $[n]$ arranged in weakly increasing order in terms of cardinality.  Let $M$ be the positive root of $ \left( 1-\frac{\lambda_1}{2} \right)\delta s-2{s\choose 2}=\frac{\lambda_1\delta s}{2}.$  (Note that $M=1+\delta-\frac{\lambda_1\delta}{2})$.  Let $t'\leq t$ denote the largest index such that $|P_{t'}|\leq M$, and let $N'=\sum_{i=1}^{t'}|P_i|$.  Then there exists a unique integer $x:=x(|P_1|,\cdots ,|P_t|)$ and $1<x^{\star}\leq M$ satisfying $N'=x+M(t'-x-1)+x^{\star}$ such that
$$ e_G(\mathcal{P})\geq\frac{1}{2}\left( \lambda_1\cdot |E(G)|+\delta x\left(1-\frac{\lambda_1}{2}\right) \right). $$
\end{lemma}

\begin{proof}

Let $\mathcal{P}_1=\{ P_i\in\mathcal{P}: |P_i|\leq M \}$.  Let $\mathcal{P}_1'$ be a partition of $\cup_{i\in P_1}P_i$ into $|\mathcal{P}_1|$ parts, each of size $1$ or $M$, with possibly one set with size $x^{\star}$, where $1\leq x^{\star}\leq M$.  Then $\sum_{i\in P_1}{|P_i|\choose 2}\leq\sum_{i\in P_1'}{|P_i'|\choose 2}.$  Let $x=|\{ P_i'\in P_1': |P_i'|=1 \}|$.  Note that $x$ is the number of parts of size one, so $x\cdot 1 +x^{\star}+(|P_1|-x)\cdot M = |\cup_{i\in P_1}P_i|$.  (Also, notice that $x\cdot 1+x^{\star}+(t-x)M\leq n$.)

Then
\begin{align*}
\sum_{P_i\in P_1}e(P_i,\overline{P}_i)&\geq \sum_{P_i\in P_1}\left[ \vol(P_i)-2{|P_i|\choose 2} \right]\\
 &\geq \sum_{P_i'\in P_1'}\left[ \vol(P_i')-2{|P_i'|\choose 2} \right] .
\end{align*}

Note that $\vol(P_i)\geq\delta |P_i|$ for each $i\in [t]$, so if $\delta |P_i|-2{|P_i|\choose 2}\geq \frac{\lambda_1}{2}\cdot\delta |P_i|$, then $\vol(P_i)-2{|P_i|\choose 2}\geq \frac{\lambda_1}{2}\cdot\vol(P_i)$, since $\frac{\lambda_1}{2}<1$.  Thus,

\begin{align*}
\sum_{P_i\in P_1}e(P_i,\overline{P}_i)&\geq  \sum_{|P_i'|=1}\left( \vol(P_i')-2{1\choose 2} \right)+\sum_{|P_i'|>1}\frac{\lambda_1}{2}\cdot \vol(P_i')\\
 &\geq \sum_{|P_i'|=1}\left( \frac{\lambda_1}{2}\cdot\vol(P_i')+\left( 1-\frac{\lambda_1}{2} \right)\vol(P_i') \right)+\sum_{|P_i'|>1}\frac{\lambda_1}{2}\cdot \vol(P_i')\\
 &\geq \sum_{i\in \mathcal{P}_1} \frac{\lambda_1}{2}\cdot\vol(P_i')+x\left( 1-\frac{\lambda_1}{2} \right)\delta \\
 &= \sum_{i\in P_1} \frac{\lambda_1}{2}\cdot\vol(P_i)+x\left( 1-\frac{\lambda_1}{2} \right)\delta .
\end{align*}

This implies that

\begin{align*}
2e(\mathcal{P})&\geq \sum_{i\in P_1} \frac{\lambda_1}{2}\cdot\vol(P_i)+x\left( 1-\frac{\lambda_1}{2} \right)\delta +\sum_{P_i\notin P_1}\frac{\lambda_1}{2}\cdot\vol(P_i)\\
 &=\sum_{i\in [t]}\frac{\lambda_1}{2}\cdot\vol(P_i)+x\left( 1-\frac{\lambda_1}{2}\right)\cdot\delta\\
  &=\lambda_1\cdot |E(G)| +x\left(1-\frac{\lambda_1}{2}\right)\cdot\delta.
\end{align*}

Thus, 

$$ e(\mathcal{P})\geq\frac{1}{2}\left( \lambda_1\cdot |E(G)|+\delta x\left(1-\frac{\lambda_1}{2}\right) \right). $$

\end{proof}

\begin{cor}\label{concavity}
Suppose $z_1, \dots, z_{t}$ form a weakly increasing sequence of positive integers.  Let $g(z)$ be as in \eqref{eq:g}, and let $M$ be the positive root of $\left( 1-\frac{\lambda_1}{2} \right)\delta s-2{s\choose 2}=\frac{\lambda_1\delta s}{2}.$.  (Note that $M=1+\delta-\frac{\lambda_1\delta}{2}$.)  Let $t' \leq t$ denote the largest index so that $z_{t'} \leq M$, and let $N'=\sum_{i=1}^{t'} z_i$.  Then there exists a unique integer $x:=x(z_1, \dots, z_{t'})$ and $1 < z^\star \leq M$ satisfying $N' = x + M(t'-x-1) + z^\star.$  Furthermore, letting $G(z_1,...,z_t)=\sum_{i=1}^t g(z_i)$,
\begin{align*}
G(z_1,...,z_t)&\geq G(\underbrace{1,...,1}_{x\text{ times}},z^{\star},\underbrace{M,...,M}_{t'-x-1\text{ times}},z_{t'+1}, z_{t'+2}, \dots, z_t)\\
 &= \delta x+g(z^{\star})+\frac{\lambda_1\delta (N-x-z^{\star})}{2}.
\end{align*}
\end{cor}

\section{Proof of Theorem \ref{t1}} \label{sec4}

Our strategy now is, in principle, simple: We use Corollary \ref{concavity} along with Lemma \ref{ColorClassSize} $(iii)$ to prove that there are sufficiently many edges leaving any partition that Lemma \ref{enoughedges} will allow us to apply Proposition \ref{RSTexistence} in each of our graphs $G_j$.  Unfortunately, while this straightforward approach works (with some effort) for partitions into not too many parts, it breaks down as the number of parts gets very close to $n$.  We handle these at the end in a slightly different way.       

\subsection{Partitions where $2\leq t\leq \frac{(1-\epsilon)}{(1+4\epsilon)}\cdot n$.}

\begin{lemma}\label{ExpectedEdges} 
For all partitions $\mathcal{P}=\{ P_1,...,P_t \}$ where $2\leq t\leq \frac{(1-\epsilon)}{(1+4\epsilon)}\cdot n$, we have that
$$(1-\epsilon)\mathbb{E}[e_{G_j}(\mathcal{P})]>\frac{(1+\epsilon)(t-2)C\log n}{2}.$$

\end{lemma}

\begin{proof}

Fix a partition $\mathcal{P} = \{P_1, \dots, P_t\}$.  Without loss of generality, assume that \\
$\vol(P_1) \leq \vol(P_2) \leq \dots \leq \vol(P_t)$.  Let $x=x(|P_1|,...,|P_t|)$ be as in Corollary \ref{concavity}.

\noindent\textbf{Case 1:} $\vol(P_t)\geq\frac{1}{2}\vol(G)$.


\begin{align*}
\frac{2}{p}\mathbb{E}\left[e_{G_j}(\mathcal{P})\right]&= \sum_{i\in [t-1]}e(P_i,\overline{P}_i)+e(P_t,\overline{P}_t) \\
& \geq \left( \sum_{i} f(P_i) \right)+ e(P_t,\overline{P}_t) \hspace{1.9in} \mbox{ where $f$ is from \eqref{eq:f}}\\
& \geq \left( \sum_{i} g(|P_i|) \right) +  e(P_t,\overline{P}_t) \hspace{1.9in} \mbox{ where $g$ is from \eqref{eq:g}}\\
 &\geq  \left[ \delta x+g(z^{\star}) +\frac{\lambda_1\delta(N-x-z^{\star})}{2}+e(P_t,\overline{P}_t) \right]\\ &\hspace{2.7in}\text{ by Corollary \ref{concavity}, where } N=n-|P_t|.
\end{align*}
Thus,
\begin{align*}
\frac{2}{p}\mathbb{E}\left[e_{G_j}(\mathcal{P})\right]  &\geq  \left[ \delta x+g(z^{\star}) +\frac{\lambda_1\delta(N-x-z^{\star})}{2}+\frac{\lambda_1}{2}\cdot\vol(\overline{P}_t) \right]\text{ by Cheeger's inequality }\\
   & \hspace{2.7in}\text{ and the fact that $\vol(P_t)\geq\frac{\vol(G)}{2}$}\\
   &\geq \left[ \delta x+g(z^{\star}) +\frac{\lambda_1\delta(N-x-z^{\star})}{2}+\frac{\lambda_1\delta}{2}\sum_{i\in [t-1]}|P_i| \right]\\
    &=\left[ \delta x+g(z^{\star}) +\frac{\lambda_1\delta(N-x-z^{\star})}{2}+\frac{\lambda_1\delta}{2}\cdot N \right]\\
     &=\left[ \delta x+g(z^{\star})-\frac{\lambda_1\delta z^{\star}}{2}+\frac{\lambda_1\delta (2N-x)}{2} \right]\\
      &\geq \left[ \delta x+\frac{\lambda_1\delta (2N-x)}{2} \right]\text{ since } g(z^{\star})\geq \frac{\lambda_1\delta z^{\star}}{2}.\\
      \end{align*}
 
Therefore,

 \begin{align*}
 \frac{2}{p}\mathbb{E}\left[e_{G_j}(\mathcal{P})\right]     &\geq \left[ \left( \delta -\frac{\lambda_1\delta}{2} \right)x+\lambda_1\delta\cdot N \right]\\
      &\geq  \lambda_1 \delta \left[ \frac{1}{2}x+ N \right].
\end{align*}

By the simple inequality $N\geq x+2(t-1-x)$, one obtains $x\geq 2(t-1)-N$.  Thus,
\begin{align*}
\frac{1}{2}x+N & \geq  \frac{1}{2}(2(t-1)-N)+N\\
 &= t-1+\frac{N}{2}.
\end{align*}

Let $N=\alpha t$.  Then $t-1+\frac{N}{2}=\left( 1+\frac{\alpha}{2} \right)t-1$.  Since $\alpha\geq 1/2$, this implies that

$$ \mathbb{E}[e(\mathcal{P})]\geq\frac{C\log n}{2}\left( \frac{5}{4}t-1 \right). $$

Therefore,
$$ (1-\epsilon)\mathbb{E}[e(\mathcal{P})]\geq (1-\epsilon)\cdot \frac{C\log n}{2}\left( \frac{5}{4}t-1 \right), $$
and for $0<\epsilon<\frac{1}{9}<\frac{t+4}{9t-4}$ we have that
$$ (1-\epsilon)\cdot \frac{C\log n}{2}\left( \frac{5}{4}t-1 \right)>\frac{(1+\epsilon)(t-2)C\log n}{2}. $$

\noindent\textbf{Case 2:} $\vol(P_t)<\frac{1}{2}\vol(G),$ and $2\leq t < \frac{(1-\epsilon)}{(1+\epsilon)}\cdot\frac{n}{2}$.

Note that this is slightly simpler than Case 1, as we may apply Cheeger's inequality directly to each part of the partition.  Observe,

\begin{align*}
2\mathbb{E}[e(\mathcal{P})]&\geq p\sum_{i\in I_1}e(P_i,\overline{P_i})\\
 &\geq  p\sum_{i\in [t]}\frac{\lambda_1\cdot\delta}{2}\cdot |P_i|\\
  &= p\cdot\frac{\lambda_1\cdot\delta}{2}\cdot n\\
   &=\frac{C\log n\cdot n}{2}.
\end{align*}

Thus, $$\mathbb{E}[e(\mathcal{P})]\geq\frac{ C\log n\cdot n}{4}.$$

Since $t\leq\frac{(1-\epsilon)}{(1+\epsilon)}\cdot\frac{n}{2}$, we have that
$$ \frac{(1-\epsilon)n}{4}>\frac{(1+\epsilon)(t-2)}{2}. $$

This implies that 
\begin{align*}
(1-\epsilon )\mathbb{E}[e(\mathcal{P})]&\geq \frac{(1-\epsilon)n\cdot C\log n}{4}\\
 &>  \frac{(1+\epsilon)(t-2)C\log n}{2}.
\end{align*}

Thus, we have that Lemma \ref{ExpectedEdges} holds in both the case where $\vol(P_t)\geq\frac{1}{2}\vol(G)$ and the case where $\vol(P_t)<\frac{1}{2}\vol(G)$ and $2\leq t<\frac{(1-\epsilon)}{(1+\epsilon)}\cdot\frac{n}{2}$.


\item \textbf{Case 3:}  For each $i\in [t]$, $\vol(P_i)<\frac{1}{2}\vol(G)$ and $\displaystyle{\frac{(1-\epsilon)}{(1+\epsilon)}\cdot\frac{n}{2}\leq t\leq\frac{(1-\epsilon)}{(1+4\epsilon)}\cdot n}$.

The following proposition follows from Corollary \ref{concavity} and will be of use to us in this case.

\begin{prop}\label{OnesProp}
Let $M$ be the value for which $\delta M-M(M-1)=\frac{\lambda\delta M}{2}$, and let $x=x(|P_1|,...,|P_t|)$ be as in Corollary \ref{concavity}.  If $t= n-y$, then $x\geq t-\lfloor\frac{y}{M-1}\rfloor -1$.
\end{prop}

Observe,
\begin{align*}
\frac{2}{p}\cdot\mathbb{E}[e(\mathcal{P})] &\geq   \delta x+g(z^{\star})+\frac{\lambda_1\delta}{2}(n-x-z^{\star}) 
\text{ by Corollary \ref{concavity}}\\
 &\geq  \left( \delta-\frac{\lambda_1\delta}{2} \right)x+\frac{\lambda_1\delta}{2}n \\
  &\geq  \frac{\delta}{2}\left( t-\left\lfloor\frac{y}{M-1}\right\rfloor -1 \right)+\frac{\lambda_1\delta}{2}n \text{ by Proposition \ref{OnesProp}}\\
   &\geq   \frac{\delta}{2}\left( t-\left(\frac{n-t}{M-1}\right) -1 \right)+\frac{\lambda_1\delta}{2}n \\
    &=\frac{\delta}{2}\left( \left( t-\left(\frac{n-t}{M-1}\right) -1 \right)+\lambda_1\cdot n \right).
\end{align*}

Let $t=\alpha\cdot n$.  Then we have that
\begin{align*}
\frac{2}{p}\cdot\mathbb{E}[e(\mathcal{P})] &\geq \frac{\delta\cdot n}{2}\left( \left( \alpha-\left(\frac{1-\alpha}{M-1}\right) -\frac{1}{n}\right)+\lambda_1 \right)\\
 &=\frac{\lambda_1\delta\cdot n}{2}\left( \frac{1}{\lambda_1}\left( \alpha-\left(\frac{1-\alpha}{M-1}\right) -\frac{1}{n}\right)+1 \right)\\
   &\geq \frac{\lambda_1\delta\cdot n}{2}\left( \alpha-\left(\frac{1-\alpha}{M-1}\right)-\frac{1}{n}+1 \right).
\end{align*}

Thus, $\displaystyle{\mathbb{E}[e(\mathcal{P})]\geq\frac{n C\log n}{4}\left( \alpha-\left(\frac{1-\alpha}{M-1}\right)-\frac{1}{n}+1 \right).}$

We want to show that $\displaystyle{(1-\epsilon )\mathbb{E}[e(\mathcal{P})]> \frac{(1+\epsilon )(t-2)C\log n}{2}}$, so it suffices to show that

$$ \frac{(1-\epsilon)n}{2}\left( \left( 1+\frac{1}{M-1} \right)\alpha -\frac{1}{M-1}-\frac{1}{n}+1 \right)>(1+\epsilon)(\alpha n-2). $$

Thus, it suffices to show 

$$ \frac{(1-\epsilon)}{2}\left( \alpha -\frac{1}{M-1}-\frac{1}{n}+1 \right)>(1+\epsilon)\alpha. $$

This is implied by

$$ \alpha <\frac{(1-\epsilon)}{(1+3\epsilon)}\left( 1-\frac{1}{M-1}-\frac{1}{n} \right). $$

Note that $$\displaystyle{\frac{(1-\epsilon)}{(1+3\epsilon)}\left( 1-\frac{1}{M-1}-\frac{1}{n} \right)=\frac{(1-\epsilon)}{(1+3\epsilon)}\left( 1-\frac{1}{\delta -\lambda_1\delta/2}-\frac{1}{n} \right)},$$
 and 
$$\displaystyle{\frac{(1-\epsilon)}{(1+3\epsilon)}\left( 1-\frac{1}{\delta -\lambda_1\delta/2}-\frac{1}{n} \right)  >\frac{(1-\epsilon)}{(1+4\epsilon)}},$$ for $$\displaystyle{n>\left( 1-\frac{(1+3\epsilon)}{(1+4\epsilon)}-\frac{2}{\delta (2-\lambda_1)} \right)^{-1}}.$$

This is an ugly expression, but note that for $\epsilon=0.1$, we have that
$$ \left( 1-\frac{(1+3\epsilon)}{(1+4\epsilon)}-\frac{2}{\delta (2-\lambda_1)} \right)^{-1}<\left( 1-\frac{1.3}{1.4}-\frac{2}{C\log n} \right)^{-1}. $$

The inequality 
$$n>\left( 1-\frac{1.3}{1.4}-\frac{2}{C\log n} \right)^{-1}$$
is satisfied for rather mild $n$ depending on $C$.  For example, if $C=100$, then $n=15$ suffices.

Therefore,
$$\displaystyle{(1-\epsilon )\mathbb{E}[e(\mathcal{P})]> \frac{(1+\epsilon )(t-2)C\log n}{2}}$$ for $n$ sufficiently large.

\end{proof}


\subsection{Partitions where $\frac{(1-\epsilon)}{(1+4\epsilon)}\cdot n\leq t< n-2$.}

\begin{lemma}  Assume that for each $i\in [t]$, $\vol(P_i)<\frac{1}{2}\vol(G)$ and $\frac{(1-\epsilon)}{(1+4\epsilon)}\cdot n<t< n-2$.  Then there are at least $t-1$ colors between parts with probability at least $1-n^{-2}$.
\end{lemma}

Notice that this is a corollary of the following lemma.

\begin{lemma}\label{ProbLemma}
Let $\Pi$ be the set of partitions of $[n]$ into $t$ parts, where $\frac{(1-\epsilon)}{(1+4\epsilon)}\cdot n<t<n-2$.  Let $C_1,...,C_s$ be the color classes of $G$, and let $c_i=|C_i|$.  Let $\mathcal{C}$ be the set of collections of $s-t+1$ color classes.  For a partition, $\mathcal{P}\in\Pi$ and $\mathfrak{C}\in\mathcal{C}$, let $\mathcal{B}_{\mathcal{P},\mathfrak{C}}$ be the event that none of the $s-t+1$ color classes in $\mathfrak{C}$ show up outside of the parts of $\mathcal{P}$.  Then

$$ \mathbb{P}\left[ \bigcup_{\mathcal{P}\in\Pi} \bigcup_{\mathfrak{C}\in\mathcal{C}} \mathcal{B}_{\mathcal{P},\mathfrak{C}} \right]\leq n^{-2}$$

for $n$ and $C$ sufficiently large.

\end{lemma}

\begin{proof}

Fix a partition, $\mathcal{P}\in \Pi$, and for each $i\in [s]$, let  $\mathcal{C}_i^{\mathcal{P}}$ be the number of edges of color $i$ between parts in $\mathcal{P}$.  Then

\begin{align*}
\mathbb{P}\left( \cup_{\mathfrak{C}\in\mathcal{C}}\mathcal{B}_{\mathcal{P},\mathfrak{C}} \right)&\leq \sum_{\mathfrak{C}\in\mathcal{C}}\mathbb{P}\left( \mathcal{B}_{\mathcal{P},\mathfrak{C}} \right)\\
 &=\sum_{\mathfrak{C}\in\mathcal{C}}\prod_{i\in I_{\mathfrak{C}}}(1-p)^{|C_i|-|C_i\cap\mathcal{P}|}\\
  &=\sum_{I\in {[s]\choose s-t}}(1-p)^{\sum_{i\in I}(|C_i|-|C_i\cap\mathcal{P}|)}\\
   &\leq \sum_{I\in {[s]\choose s-t}}\exp(-p\sum_{i\in I}(|C_i|-|C_i\cap\mathcal{P}|))\\
 &=\sum_{I\in {[s]\choose s-t}}\exp\left( -p\sum_{i\in I}\mathcal{C}_i^{\mathcal{P}} \right).
\end{align*}

So we want to consider $\sum_{I\in {[s]\choose s-t}}\exp\left( -p\sum_{i\in I}\mathcal{C}_i^{\mathcal{P}} \right)$.  Let 
$$ f\left(\mathcal{C}_1^{\mathcal{P}},...,\mathcal{C}_s^{\mathcal{P}};t\right)=\sum_{I\in {[s]\choose s-t}}\exp\left( -p\sum_{i\in I}\mathcal{C}_i^{\mathcal{P}} \right). $$

We begin with two observations:
 
\noindent\textbf{Claim 1.}  
$$f\left(\mathcal{C}_1^{\mathcal{P}},\cdots \mathcal{C}_s^{\mathcal{P}};t\right)\leq f\left(\underbrace{0,\cdots 0}_{k-1\text{ times}},x^{*},\underbrace{\frac{\lambda_1\delta}{2},\cdots ,\frac{\lambda_1\delta}{2}}_{s-k\text{ times}};t\right),$$
where $0\leq x^{*}<\frac{\lambda_1\delta}{2}$ and $(s-k)\frac{\lambda_1\delta}{2}+x^{*}=e(\mathcal{P})$.

\noindent\textbf{Claim 2.}
$$ f\left( \underbrace{0,\cdots ,0}_{k-1\text{ times}},x^{*},\underbrace{\frac{\lambda_1\delta}{2},\cdots ,\frac{\lambda_1\delta}{2}}_{s-k\text{ times}};t \right)\leq f\left( \underbrace{0,\cdots ,0}_{k\text{ times}},\underbrace{\frac{\lambda_1\delta}{2},\cdots ,\frac{\lambda_1\delta}{2}}_{s-k\text{ times}};t \right).$$

The first of these follows from convexity (details are as in \cite{CHH}), while the second follows from monotonicity of $f$ in its variables.  

Assume that $t=n-y$ for $y\geq 3$, and let $x$ be as in Lemma \ref{EdgesConcavity}.

Notice that 
\begin{align}
s-k&\geq \frac{e(\mathcal{P})}{\lambda_1\delta/2}\nonumber\\
 &\geq  \frac{|E(G)|}{\delta}+\frac{x}{\lambda_1}\left(1-\frac{\lambda_1}{2}\right)\text{ by Lemma \ref{EdgesConcavity}}\nonumber\\
  &\geq \frac{|E(G)|}{\delta}+\left( t-\frac{y}{M-1}-1 \right)\left(1-\frac{\lambda_1}{2}\right)\cdot\frac{1}{\lambda_1}\text{ by Proposition \ref{OnesProp}}\nonumber\\
   &\geq \frac{|E(G)|}{\delta}+\frac{1}{2}\left( t-\frac{n-t}{M-1}-1 \right) \text{ as } \left(1-\frac{\lambda_1}{2}\right)\cdot\frac{1}{\lambda_1}\geq \frac{1}{2}.\label{eq:star}
\end{align}

Then

\begin{align}
f\left(\mathcal{C}_1^{\mathcal{P}},...,\mathcal{C}_s^{\mathcal{P}};t\right)&=\sum_{I\in {[s]\choose s-t}}\exp\left( -p\sum_{i\in I}\mathcal{C}_i^{\mathcal{P}} \right)\nonumber\\
  &\leq \sum_{r=\max\{0,t-(s-k)\}}^{\min\{ t,k \}}{k\choose r}{s-k\choose s-t-k+r}\exp\left( -p(s-k-t+r)\cdot\frac{\lambda_1\delta}{2} \right).\label{eq:CC1}
\end{align}

Here we are choosing $s-t$ colors to vanish.  First, we choose $r$ that we will not take, which means that we will take $k-r$.  Then we need to choose the remaining $s-t-(k-r)=s-t-k+r$ colors.  This gives us the following:

\begin{align}
\eqref{eq:CC1}  &\leq \sum_{r=\max\{0,t-(s-k)\}}^{\min\{ t,k \}}\exp\left( r\log k+(s-k-t+r)\log(s-k)-\frac{C\log n}{2}(s-k-t+r) \right)\nonumber\\
    &\leq  n\cdot \exp\left( (s-k-t)\log(s-k)-\frac{C\log n(s-k-t)}{2} \right)\nonumber\\
     &\leq  \exp\left( \log n+(s-k-t)\left( \log(n^2)-\frac{C\log n}{2} \right) \right)\nonumber\\
      &= \exp\left( \log n\left( 1+(s-k-t)\left(2-\frac{C}{2}\right)  \right) \right)\nonumber\\
       &=\exp\left( \log n\left( 1-(s-k-t)\left(\frac{C-4}{2}\right) \right) \right).\label{eq:CC}
\end{align}
Now we use \eqref{eq:star} to continue:
\begin{align}
  \eqref{eq:CC}     &\leq \exp\left( \log n\left( 1-\left( \frac{|E(G)|}{\delta}+\frac{1}{2}\left( t-\frac{n-t}{M-1}-1 \right)-t \right)\left( \frac{C-4}{2} \right) \right) \right)\nonumber\\
         &\leq \exp\left( \log n\left( 1-\left( \frac{n}{2}+\frac{1}{2}\left( t-\frac{y}{M-1}-1 \right)-t \right)\left( \frac{C-4}{2} \right) \right) \right)\nonumber\\
          & ~~~~~~~~~~~~~~~~~~~~~~~~~~~~~~~~~~~~~~~~~~~~~~~~~~~~~~~~~~\text{ since } |E(G)|\geq n\cdot\delta/2\nonumber\\
          &=\exp\left( \log n\left( 1-\left( \frac{n-t}{2}-\frac{y}{2(M-1)}-\frac{1}{2} \right)\left( \frac{C-4}{2} \right) \right) \right)\nonumber\\
           &=\exp\left( -\log n\left( \left( \frac{y}{2}-\frac{y}{2(M-1)}-\frac{1}{2}  \right)\left( \frac{C-4}{2} \right) -1 \right) \right)\nonumber\\
            &=\exp\left( -\log n\left( \left( y\left( \frac{1}{2}-\frac{1}{2(M-1)} \right)-\frac{1}{2} \right)\left( \frac{C-4}{2} \right) -1 \right) \right)\nonumber\\
             &\leq \exp\left( -\log n\left( \left( \frac{y}{4}-\frac{1}{2} \right)\left( \frac{C-4}{2} \right)-1 \right) \right).\nonumber
\end{align}

Observe that the number of partitions of the vertices into $t$ parts is
\begin{align*}
{n\choose t}\cdot t^{n-t}&={n\choose n-t}\cdot t^{n-t}\\
 &\leq  n^{n-t}\cdot t^{n-t}\\
  &\leq \exp(2\cdot (n-t)\log (n))\\
   &=\exp(2y\log (n)).
\end{align*}

Thus, 
\begin{align*}
\mathbb{P}\left[ \bigcup_{\mathcal{P}\in\Pi} \bigcup_{\mathfrak{C}\in\mathcal{C}} \mathcal{B}_{\mathcal{P},\mathfrak{C}} \right]&\leq  \exp\left( 2y\log (n)-\log n\left( \left( \frac{y}{4}-\frac{1}{2} \right)\left( \frac{C-4}{2} \right)-1 \right) \right)\\
 &=\exp\left( -\log n\left( \left( \frac{y-2}{4} \right)\left( \frac{C-4}{2} \right)-1-2y \right) \right).
\end{align*}

Notice that $\left( \frac{y-2}{4} \right)\left( \frac{C-4}{2} \right)-1-2y \geq 2$ for $y\geq 3$ and $C$ sufficiently large.  Thus, for $y\geq 3$, we have
$$ \mathbb{P}\left[ \bigcup_{\mathcal{P}\in\Pi} \bigcup_{\mathfrak{C}\in\mathcal{C}} \mathcal{B}_{\mathcal{P},\mathfrak{C}} \right]\leq n^{-2}. $$

\end{proof}

\subsection{Partitions where $n-2\leq t\leq n$.}

\begin{lemma} Assume that for each $i\in [t]$, $\vol(P_i)<\frac{1}{2}\vol(G)$.  Let $t\in \{ n-2, n-1,n \}$.  Then there are at least $t-1$ colors between parts with probability at least $1-n^{-2}$.
\end{lemma}

\begin{proof}

Fix a partition with $t$ parts where $t\in \{n-2,n-1,n\}$.  We want to show that there are at least $t-1$ colors between the parts of our partition.  Unfortunately it is impossible to prove a lower bound on the number of edges in color class $C_i$ in some $G_j$ since $C_i$ may be too small.  To circumvent this, instead of considering each individual color class, we combine color classes to create pseudocolor classes.  As shown in \cite{CHH}, we can construct $n-1$ pseudocolor classes $D_1,...,D_{n-1}$ such that for each $k\in [n-1]$, 
$$ D_k=\left( \cup_{j=1}^{\ell} D_j \right)\backslash\left( \cup_{i=1}^{k-1} D_i \right), $$
where $\ell$ is the least integer such that $\left| \left( \cup_{j=1}^{\ell} C_j \right)\backslash\left( \cup_{i=1}^{k-1} D_i \right) \right|\geq n/4$.  Fix $i\in [n-1]$ and $j\in [t]$.  Let $Z_i^{(j)}=|E(G_i)\cap D_j|$.  Then
$$ \mathbb{E}[Z_i^{(j)}]\geq p\cdot\frac{n}{4}=\frac{n\cdot C\log n}{4\lambda_1\delta}\geq\frac{n\cdot C\log n}{4\delta}\geq\frac{C}{4}\log n. $$

Observe,
\begin{align*}
\mathbb{P}\left(Z_i^{(j)}\leq \frac{C}{8}\log n \right)&\leq \mathbb{P}\left( Z_i^{(j)}\leq\frac{1}{2}\mathbb{E}[Z_i^{(j)}] \right)\\
 &\leq \exp\left( -\frac{1}{8}\cdot\mathbb{E}[Z_i^{(j)}]\right)\\
  &\leq \exp\left( -\frac{1}{8}\cdot\frac{C}{4}\log n \right)\\
   &= \exp\left( -\frac{C\log n}{32} \right)\\
      &\leq  n^{-4}
\end{align*}
for $C$ and $n$ sufficiently large.  Thus,

\begin{align*}
\bigcup_{j\in [t]}\bigcup_{i\in [n-1]}\mathbb{P}\left(Z_i^{(j)}\leq \frac{C}{8}\log n \right) &\leq  t(n-1)\cdot n^{-4}\\
 &< n^2\cdot n^{-4}\\
  &= n^{-2}.
\end{align*} 

This shows that in each $G_j$ there are at least $\frac{C}{8}\log n$ edges within each psuedocolor class with probability at least $1-n^{-2}$.  If $t=n$, then each part in the partition consists of a single vertex, so none of the egdes can be contained within the parts.  If $t=n-1$, then there is one part of size two and the rest are of size one.  In this case there is at most one edge within the parts.  If $t=n-2$, then there are either two parts of size $2$ and the rest of size $1$ or there is one part  of size $3$ and the rest have size $1$.  Thus, there are at most three edges contained within the parts.  

Therefore, for $n-2\leq t\leq n$, there are at least $\log n$ edges within each of the $n-1$ pseudocolor classes left between parts of a partition with probability at least $1-n^{-2}$.  Since $n-1\geq t$, we have that there are at least $t-1$ colors between parts with probability at least $1-n^{-2}$.

\end{proof}

\section{Applications and Discussion} \label{sec5}

While Theorem \ref{t1} applies to all sufficiently large graphs (as a function of $\lambda_1$) it is strongest when $\lambda_1$ is close to one.  This is when the requirements on the color classes are weakest and the conclusion is strongest.  Fortunately there are some graph classes satisfying this.  The only graphs with $\lambda_1=1$ are complete bipartite graphs.  The corollary below follows immediately from Theorem \ref{t1} since $\lambda_1(K_{n,m})=1$.

\begin{cor}
Let $G$ be an edge-colored copy of $K_{n,m}$ where $m\geq n$ and $n\geq C\log (n+m)$ for $n,m,$ and $C$ sufficiently large in which each color appears on at most $n/2$ edges.  Then $G$ contains at least $\lfloor \frac{n}{C\log (n+m)} \rfloor$ edge-disjoint rainbow spanning trees.
\end{cor}

While non-complete graphs have $\lambda_1<1$, there are several natural classes of graphs which have $\lambda_1$ close to one.  First, consider random $d$-regular graphs.  Friedman, Kahn, and Szemer\'{e}di gave a bound on the eigenvalues of such graphs with $d$ fixed in \cite{FriedmanKahnSzemeredi}.  This was a combination of two papers -- one by Friedman, and the other by Kahn and Szemer\'{e}di.  Their techniques were different, and in \cite{BroderFriezeSuenUpfal}, Broder, Frieze, Suen, and Upfal showed that Kahn and Szemer\'{e}di's technique could be applied to more dense random $d$-regular graphs.  More recently, Cook, Goldstein, and Johnson improved the range at which the eigenvalue bound was known.

\begin{theorem}\label{randomdreg} (\cite{CookGoldsteinJohnson})
Let $A$ be the adjacency matrix of a uniform random $d$-regular graph on $n$ vertices.  Let $\lambda_0(A)\geq\cdots\geq\lambda_{n-1}(A)$ be the eigenvalues of $A$, and let $\lambda(A)=\max\{\lambda_1(A),-\lambda_{n-1}(A)\}$.  For any $C_0,K>0$, there exists $\alpha>0$ such that if $1\leq d\leq C_0(n^{2/3})$, then $\mathbb{P}(\lambda(A)\leq \alpha\sqrt{d})\geq 1-n^{-K}$ for $n$ sufficiently large.
\end{theorem}

In a $d$-regular graph, we have that $\lambda_1(\mathcal{L})=1-\frac{1}{d}\lambda_1(A)$.  Therefore, this result gives us a lower bound on $\lambda_1(\mathcal{L})$, which we can use to apply Theorem \ref{t1}.

\begin{cor}
Let $G$ be an edge-colored uniform random $d$-regular graph in which \\
$C\log n\leq d\leq C \cdot n^{2/3}$ (for $C$ and $n$ sufficiently large).  Then there exists $\alpha>0$ such that if each color class has size at most $d\cdot\left( 1-\frac{\alpha}{\sqrt{d}}\right)/2$, then $G$ contains at least $\left\lfloor \frac{d-\alpha\sqrt{d}}{C\log n} \right\rfloor$ edge-disjoint rainbow spanning trees with high probability.
\end{cor}

Our result applies to some graphs with very skewed degree distributions.  The graph $G_{n,p}$ is the graph on $n$ vertices in which each edge appears with probability $p$.  This can be generalized in the following way.  For a sequence $\mathbf{w}=(w_1,\cdots ,w_n)$, let $\rho=\frac{1}{\sum_{i=1}^nw_i}$.  Then $G(\mathbf{w})$ is a random graph in which we label the vertices $v_1,...,v_n$, and the edge $v_iv_j$ appears with probability $w_iw_j\rho$ (\cite{ComplexGraphsandNetworks}).  (Here, we allow for loops.)  In the graph $G(\mathbf{w})$, it is easy to see that $\mathbb{E}[\deg(v_i)]=\omega_i$.  Notice that if we take $\mathbf{w}=(np,\cdots np)$, we get $G(\mathbf{w})=G_{n,p}$.

It is well known (\cite{FurediKomlos}) that $\lambda_1(G_{n,p})\geq 1-\frac{2}{\sqrt{np}}$ with high probability.  So for all $\epsilon>0$, $\lambda_1(G_{n,p})\geq 1-\epsilon$ for $n$ sufficiently large.  Also, $\delta(G_{n,p})\geq (1-\epsilon )np$ if $np\gg \log^2n$.  These results also apply to irregular graphs.  For instance, consider $G(\mathbf{w})$.  Fix $\epsilon>0$.  If $\mathbf{w}_{\min}\gg\log^2n$, then $\delta\geq (1-\epsilon)\mathbf{w}_{\min}$ with high probability for $n$ large enough.  Also, if $\mathbf{w}_{\min}\gg \log^2 n$, then $ \lambda_1(G(\mathbf{w}))\geq 1-\epsilon $ with high probability.  This is implied by the following result of Chung, Lu, and Vu.


\begin{theorem} (\cite{ChungLuVu})
For a random graph with given expected degrees, if the minimal expected degree $\mathbf{w}_{\min}$ satisfies $\mathbf{w}_{\min}\gg\log^2 n$, then almost surely the eigenvalues of the Laplacian satisfy
$$ \max_{i\neq 0} |1-\lambda_i|\leq (1+o(1))\frac{4}{\sqrt{\overline{\mathbf{w}}}}+\frac{g(n)\log^2 n}{w_{\min}}, $$
where $\overline{\mathbf{w}}=\frac{\sum_{i=1}^n w_i}{n}$ is the average expected degree and $g(n)$ is a function tending to infinity (with $n$) arbitrarily slowly.

\end{theorem}

This bound on $\lambda_1(G(\mathbf{w}))$ gives us the following corollary of Theorem \ref{t1}.

\begin{cor}
Fix $\epsilon>0$.  Assume that $\mathbf{w}_{\min}\gg \log^2 n$ and $G(\mathbf{w})$ is edge-colored so that each color class has size at most $\frac{\mathbf{w}_{\min}\cdot\left( 1-\epsilon\right)}{2}$.  Then for $n$ and $C$ sufficiently large, a graph $G\in G(\mathbf{w})$ contains at least $\left\lfloor \frac{\mathbf{w}_{\min}\cdot\left( 1-\epsilon \right)}{C\log n} \right\rfloor$ edge-disjoint rainbow spanning trees with probability $1-o(1)$.
\end{cor}


One can also wonder about the sharpness of our result and the dependence on $\lambda_1$.  If $\lambda_1$ is small, then Cheeger's inequality implies there is a sparse cut, which limits the number of disjoint spanning trees the graph can contain, let alone the number of disjoint rainbow spanning trees.  On the other hand, certainly no more than $\delta$ edge disjoint spanning trees are possible in any graph with minimum degree $\delta$.  It seems plausible that the logarithmic factor could be removed in our lower bound on the number of rainbow spanning trees.  In the more specialized situation of proper edge colorings of $K_n$, this is what \cite{Horn} does.  

In our proof of the main result, $\lambda_1$ was mostly used to lower bound the isoperimetric constant by Cheeger's inequality.  It seems likely that $\lambda_1$ could be replaced by the isoperimetric constant.  However, the isoperimetric constant is practically impossible to compute so stating the hypothesis in terms of $\lambda_1$ seems most natural.    

It's also possible that the bound on the size of the color classes could be improved.  It is clear that if color classes are allowed to be larger than $\frac{\delta}{2}$ in a $\delta$ regular graph, that rainbow spanning trees can be avoided entirely.  In particular, for complete bipartite graphs (where $\lambda_1=1$) our size bound on color classes is correct.  It is less clear that the factor of $\lambda_1$ appearing in our bound is actually required.  We suspect that this dependence can be somewhat weakened, though it is unclear to us exactly what the dependence (if any!) on $\lambda_1$ the size of the color classes should have. Our proof rather naturally leads to bounding color classes by $\frac{\lambda_1 \delta}{2}$ as we have done.


%
%


\begin{thebibliography}{10}

\bibitem{AkbariAlipour}
S.~Akbari and A.~Alipour.
\newblock Multicolored trees in complete graphs.
\newblock {\em J. Graph Theory}, 54(3):221--232, 2007.

\bibitem{AlbertFriezeReed}
Michael Albert, Alan Frieze, and Bruce Reed.
\newblock Multicoloured {H}amilton cycles.
\newblock {\em Electron. J. Combin.}, 2:Research Paper 10, approx.\ 13 pp.\,
  1995.

\bibitem{BroderFriezeSuenUpfal}
Andrei~A. Broder, Alan~M. Frieze, Stephen Suen, and Eli Upfal.
\newblock Optimal construction of edge-disjoint paths in random graphs.
\newblock In {\em Proceedings of the {F}ifth {A}nnual {ACM}-{SIAM} {S}ymposium
  on {D}iscrete {A}lgorithms ({A}rlington, {VA}, 1994)}, pages 603--612. ACM,
  New York, 1994.

\bibitem{BroersmaLi}
Hajo Broersma and Xueliang Li.
\newblock Spanning trees with many or few colors in edge-colored graphs.
\newblock {\em Discuss. Math. Graph Theory}, 17(2):259--269, 1997.

\bibitem{BrualdiHollingsworth}
Richard~A. Brualdi and Susan Hollingsworth.
\newblock Multicolored trees in complete graphs.
\newblock {\em J. Combin. Theory Ser. B}, 68(2):310--313, 1996.

\bibitem{BrualdiHollingsworth2}
Richard~A. Brualdi and Susan Hollingsworth.
\newblock Multicolored forests in complete bipartite graphs.
\newblock {\em Discrete Math.}, 240(1-3):239--245, 2001.

\bibitem{CarraherHartke}
James~M. Carraher and Stephen~G. Hartke.
\newblock Eulerian circuits with no monochromatic transitions in edge-colored
  digraphs.
\newblock {\em SIAM J. Discrete Math.}, 27(4):1924--1939, 2013.

\bibitem{CHH}
James~M. Carraher, Stephen~G. Hartke, and Paul Horn.
\newblock Edge-disjoint rainbow spanning trees in complete graphs.
\newblock {\em European J. Combin.}, 57:71--84, 2016.

\bibitem{Chernoff}
Herman Chernoff.
\newblock A note on an inequality involving the normal distribution.
\newblock {\em Ann. Probab.}, 9(3):533--535, 1981.

\bibitem{ComplexGraphsandNetworks}
Fan Chung and Linyuan Lu.
\newblock {\em Complex graphs and networks}, volume 107 of {\em CBMS Regional
  Conference Series in Mathematics}.
\newblock Published for the Conference Board of the Mathematical Sciences,
  Washington, DC; by the American Mathematical Society, Providence, RI, 2006.

\bibitem{ChungLuVu}
Fan Chung, Linyuan Lu, and Van Vu.
\newblock The spectra of random graphs with given expected degrees.
\newblock {\em Internet Math.}, 1(3):257--275, 2004.

\bibitem{Chung}
Fan R.~K. Chung.
\newblock {\em Spectral graph theory}, volume~92 of {\em CBMS Regional
  Conference Series in Mathematics}.
\newblock Published for the Conference Board of the Mathematical Sciences,
  Washington, DC; by the American Mathematical Society, Providence, RI, 1997.

\bibitem{Constantine}
Gregory~M. Constantine.
\newblock Edge-disjoint isomorphic multicolored trees and cycles in complete
  graphs.
\newblock {\em SIAM J. Discrete Math.}, 18(3):577--580, 2004/05.

\bibitem{CookGoldsteinJohnson}
N.~Cook, L.~Goldstein, and T.~Johnson.
\newblock Size biased couplings and the spectral gap for random regular graphs.
\newblock {\em arxiv:1510.06013 [math.PR]}, 2015.

\bibitem{Edmonds}
Jack Edmonds.
\newblock Submodular functions, matroids, and certain polyhedra.
\newblock In {\em Combinatorial {S}tructures and their {A}pplications ({P}roc.
  {C}algary {I}nternat. {C}onf., {C}algary, {A}lta., 1969)}, pages 69--87.
  Gordon and Breach, New York, 1970.

\bibitem{FriedmanKahnSzemeredi}
J.~Friedman, J.~Kahn, and E.~Szemer\'{e}di.
\newblock On the second eigenvalue of random regular graphs.
\newblock In {\em Proceedings of the 21st Annual ACM Symposium on Theory of
  Computing}, number~3, pages 587--598, 1989.

\bibitem{FriezeKrivelevich}
Alan Frieze and Michael Krivelevich.
\newblock On rainbow trees and cycles.
\newblock {\em Electron. J. Combin.}, 15(1):Research paper 59, 9, 2008.

\bibitem{FuLoPerryRodger}
Hung-Lin Fu, Yuan-Hsun Lo, K.E. Perry, and C.A. Rodger.
\newblock On the number of rainbow spanning trees in edge-colored complete
  graphs.
\newblock {\em Preprint}, 2016.

\bibitem{FurediKomlos}
Z.~F\"uredi and J.~Koml\'os.
\newblock The eigenvalues of random symmetric matrices.
\newblock {\em Combinatorica}, 1(3):233--241, 1981.

\bibitem{Horn}
Paul Horn.
\newblock Rainbow spanning trees in complete graphs colored by one
  factorizations.
\newblock {\em Preprint}, 2013.

\bibitem{KKS}
A.~Kaneko, M.~Kano, and K.~Suzuki.
\newblock Three edge disjoint multicolored spanning trees in complete graphs.
\newblock {\em Preprint}, 2003.

\bibitem{KanoLi}
Mikio Kano and Xueliang Li.
\newblock Monochromatic and heterochromatic subgraphs in edge-colored
  graphs---a survey.
\newblock {\em Graphs Combin.}, 24(4):237--263, 2008.

\bibitem{KMV}
John Krussel, Susan Marshall, and Helen Verrall.
\newblock Spanning trees orthogonal to one-factorizations of {$K_{2n}$}.
\newblock {\em Ars Combin.}, 57:77--82, 2000.

\bibitem{Schrijver}
Alexander Schrijver.
\newblock {\em Combinatorial optimization. {P}olyhedra and efficiency. {V}ol.
  {B}}, volume~24 of {\em Algorithms and Combinatorics}.
\newblock Springer-Verlag, Berlin, 2003.
\newblock Matroids, trees, stable sets, Chapters 39--69.

\bibitem{Suzuki}
Kazuhiro Suzuki.
\newblock A necessary and sufficient condition for the existence of a
  heterochromatic spanning tree in a graph.
\newblock {\em Graphs Combin.}, 22(2):261--269, 2006.

\end{thebibliography}

\end{document}